\documentclass{amsart}

\usepackage{amsmath, amsfonts, amsthm, amssymb, verbatim}

\usepackage[mathcal]{euscript}

\newtheorem{thm}{Theorem}[section]
\newtheorem{cor}[thm]{Corollary}
\newtheorem{lem}[thm]{Lemma}
\newtheorem{prop}[thm]{Proposition}
 \newtheorem{defn}[thm]{Definition}

 \newtheorem{rem}[thm]{Remark}

 \numberwithin{equation}{section}

 \newcommand {\Phim} {\Phi_\eta}

 \newcommand {\Thetam}{\Theta_\eta}

 \newcommand {\veps}{v_\varepsilon}

 \newcommand {\vepsn}{v_{\varepsilon,n}}

 \newcommand {\veta} {v_{\varepsilon}^{\eta}}

 \newcommand{\QQ}{\mathbb{R}^N\times(0,\infty)}

  \newcommand{\set}[1]{\left\{#1\right\}}

\begin{document}

%\linenumbers

\title[Hamilton-Jacobi Equations]{Sharp decay estimates and vanishing viscosity\\ for diffusive Hamilton-Jacobi equations}

\author{ Sa\"\i d Benachour \and Matania Ben-Artzi \and Philippe Lauren\c{c}ot }

\address{Sa\"\i d Benachour: Institut Elie Cartan , Universit\'{e} Henri
Poincar\'e,  F-54506 Vand\oe uvre-l\`{e}s-Nancy Cedex, France.} 
\email{Said.Benachour@iecn.u-nancy.fr}

\address{Matania Ben-Artzi: Institute of Mathematics, Hebrew
University of Jerusalem, Jerusalem 91904, Israel.} 
\email{mbartzi@math.huji.ac.il}

\address{Philippe Lauren\c{c}ot: Institut de Math\'{e}matiques de
Toulouse, CNRS UMR~5219, Universit\'{e} de Toulouse, 118 Route de
Narbonne, F--31062 Toulouse Cedex 9, France.} 
\email{Philippe.Laurencot@math.univ-toulouse.fr}

\thanks{M. Ben-Artzi thanks the Institut Elie Cartan , Universit\'{e}
de Nancy I, for the hospitality during February 2000, when this work
was started.} 



\date{\today}

%%% ----------------------------------------------------------------------

\begin{abstract}
Sharp temporal decay estimates are established for the gradient and time derivative of solutions to the Hamilton-Jacobi equation
$\partial_t \veps + H(|\nabla_x \veps|)=\varepsilon\ \Delta \veps$ in $\QQ$, the parameter $\varepsilon$ being either positive or zero. Special care is given to the dependence of the estimates on $\varepsilon$. As a by-product, we obtain convergence of the sequence $(\veps)$ as $\varepsilon\to 0$ to a viscosity solution, the initial condition being only continuous and either bounded or non-negative. The main requirement on $H$ is that it grows superlinearly or sublinearly at infinity, including in particular $H(r)=r^p$ for $r\in [0,\infty)$ and $p\in (0,\infty)$, $p\ne 1$.
\end{abstract}

%%% ----------------------------------------------------------------------

\maketitle

%%% ----------------------------------------------------------------------

\section{\textbf {Introduction}}

The purpose of this paper is to derive temporal decay estimates for the gradient and the time derivative of viscosity solutions to the Hamilton-Jacobi equation
\begin{eqnarray}
\label{HJ}
\partial_t v + H(|\nabla_x v|) & = & 0,  \qquad (x,t) \in \QQ ,\\
\label{initial}
v(x,0) & = & \varphi(x),  \qquad  x\in \mathbb{R}^N.
\end{eqnarray}
and its diffusive counterpart
\begin{eqnarray}
\partial_t \veps-\varepsilon \Delta \veps+H(|\nabla \veps|) & = & 0,
\qquad (x,t) \in \QQ , \quad \varepsilon>0, \label {HJV}\\ 
\veps(x,0) & = & \varphi(x),\qquad  x\in \mathbb{R}^N,
\label{init}
\end{eqnarray}
under suitable assumptions on the Hamiltonian function $H$ and for initial data $\varphi$ which are continuous but not necessarily uniformly continuous (and in some cases not even bounded). The main feature of our analysis is that we carefully trace the dependence on the ``viscosity'' parameter $\varepsilon$ in the estimates of the space-time gradients of $\veps$. We obtain estimates which do not deteriorate as $\varepsilon\to 0$ and thus reflect the regularizing effect of the nonlinear term $H(|\nabla \veps|)$. As a by-product of our analysis, we may perform the limit $\varepsilon\to 0$ (the so-called vanishing viscosity limit) and show the convergence of the solutions $v_\varepsilon$ to the nonlinear parabolic equation \eqref{HJV}-\eqref{init} without requiring much on the initial condition (besides continuity and either boundedness or only non-negativity). The limiting solutions we obtain are ``viscosity solutions'' in the sense of Crandall \& Lions \cite{CL83}, and we refer  to \cite{BCD97,Ba94,Cr97,Li82} for extensive discussions of these solutions and to \cite[Chapter~10]{Ev98} for the connection between viscosity solutions and the ``vanishing viscosity'' approach.

The main tool in this work consists of uniform (with respect to
$\varepsilon$) estimates of the (space-time) gradient of
$\veps$. These estimates enable us to treat the more general initial
data $\varphi$ as mentioned above. Roughly speaking, the main requirement placed on our Hamiltonian function
$H=H(r)$, $0\leq r<\infty,$ is that it grows either ``superlinearly'' or
``sublinearly'' as $r \rightarrow \infty$. More precisely,  the basic set of assumptions \eqref{Hass}-\eqref{Hass3} on $H$ is the following. 

\begin{equation}
\label{Hass}
\begin{cases}
& \bullet\ H\quad \mbox{is continuous nonnegative on}\quad
[0,\infty)\quad \mbox{and}\quad H(0)=0.\\ 
& \mbox{In addition,} \quad H\quad \mbox{is locally Lipschitz
continuous on} \quad(0,\infty).\\ 
& \bullet\ \mbox{There exists a family of nonnegative smooth functions}\\
& \{\Phim\}_{\eta>0} \quad\mbox{defined in}\quad [0,\infty)\quad
\mbox{such that}\\
(i)&\quad \Phim(0)=0 \quad \mbox{for all} \quad \eta>0.\\
(ii)&\quad \Phim(r^2) \xrightarrow[\eta\rightarrow
{0}^{+}]{}H(r),\quad \mbox{uniformly  in compact intervals of}\quad
[0,\infty). 
 \end{cases}
 \end{equation}

\begin{defn}
 \label{pcondi}
 Consider the family of functions $\{\Thetam\}_{\eta>0}$ defined by
$$
\Thetam(r)=2r\Phim'(r)-\Phim(r),\qquad (r,\eta)\in [0,\infty)\times (0,\infty).
$$
Let $p\in (0,1)\cup(1,\infty).$ We say that $H$ satisfies the
\textbf{$p$-condition} if there exist $\gamma>0,\quad a>0, \quad b>0,
$ such that, for $r>0$ and sufficiently small $\eta>0,$ 
$$
\begin{cases}
(i)\quad \Thetam(r)\geq
ar^{\frac{p}{2}}-b\eta^\gamma,\quad\mbox{if}\quad p\in(1,\infty),\\ 
(ii) \quad \Thetam(r) \leq -ar^{\frac{p}{2}}+b\eta^\gamma
,\quad\mbox{if}\quad p\in(0,1).  
\end{cases}
$$
\end{defn}

Our third basic assumption is
\begin{equation}
\label{Hass3}
\qquad \mbox{ For some $p\in (0,1)\cup(1,\infty)$,}\quad H \quad
\mbox{satisfies the $p$-condition}. 
\end{equation}

As we show in  Appendix~\ref{apA} to this paper, the prototypical example

\begin{equation}
\label{special}
H(r)=r^p,\quad p\in (0,1)\cup (1,\infty),\quad 
\end{equation}
satisfies the above assumptions with
$\Phim(r)=(r+\eta^2)^{\frac{p}{2}}-\eta^p \quad.$ In fact, the same
argument shows that one can take 
\begin{equation}
\label{specialpl}
H(r)=\sum\limits_{k=1}^m \mu_k r^{p_k},\quad \mu_k>0,
\end{equation}
where either $\set{p_1,...,p_m}\in (0,1)^m,$ or $\set{p_1,...,p_m}\in
(1,\infty)^m.$ 

We can easily extend further this special case, as shown by the
following example. 

 \begin{prop}
Let $p>1$ and let $G$ be a smooth function supported in $[r_0,\infty)$
for some $r_0>0$. Assume that for some $q\geq p$ and $\lambda>0$ we have 
$$
\frac{d}{dr}\left( \frac{G(r)}{r} \right)\geq \lambda r^{q-2},\quad r>r_0.
$$
Then the function $H~: r \longmapsto r^p+G(r)$  satisfies all the
assumptions~\eqref{Hass}-\eqref{Hass3}. In particular, we can take
$G(r)=(r-r_0)_+^q.$ 
 \end{prop}

\begin{proof}
It suffices, by the above remarks, to consider only the part of
$G$. By taking $\Phim(r^2)\equiv G(r)$ we get 
$$\Thetam(r)=2r\Phim'(r)-\Phim(r)=\sqrt{r}G'(\sqrt{r})-G(\sqrt{r})\geq
\lambda (\sqrt{r})^q\geq \lambda\ r_0^{\frac{q-p}{2}}\
r^{\frac{p}{2}}.$$ 
\end{proof}

From now on, we assume that all special Hamiltonians $H$ satisfy a
$p$-condition for some $p\in (0,1)\cup (1,\infty)$. Our initial
function $\varphi$ is assumed only to be bounded from below and can be taken in $C(\mathbb{R}^N)$, the  space of real-valued continuous functions on  $\mathbb{R}^N$ if $1<p<\infty$ whereas, if $0<p<1$ , it
is taken in $C_b(\mathbb{R}^N)$, the space of \textit{bounded}
continuous  functions . 

Under these conditions we obtain the existence and uniqueness of a solution to \eqref{HJ}-\eqref{initial} in $\mathbb{R}^N \times
[0,\infty)$ in Theorems~\ref{bounded} and~\ref{general}, provided a suitable comparison principle is available in the following sense. 

\begin{defn}\label{comparison} (a) We say that Equation~\eqref{HJ}
satisfies the \textbf{(discontinuous) comparison principle} if the
following condition holds: 
  Let $v_1\in USC(\QQ)$ (resp. $v_2\in LSC(\QQ)$) be a viscosity
subsolution  (resp. supersolution) of ~\eqref{HJ}. Assume that
$v_1(x,0)\leq v_2(x,0)$ for $x\in\mathbb{R}^N$ and that 
$\inf_{\QQ} v_2>-\infty.$ Then $v_1\leq v_2$ in $\QQ.$

(b) We say that Equation~\eqref{HJ} satisfies the \textbf{comparison
principle in $C_b(\mathbb{R}^N)$} if the following condition holds: 
  Let $v_1\in C_b(\QQ)$ (resp. $v_2\in C_b(\QQ)$) be a viscosity
subsolution  (resp. supersolution) of ~\eqref{HJ}. Assume that
$v_1(x,0)\leq v_2(x,0)$ for $x\in\mathbb{R}^N$. Then $v_1\leq v_2$ in
$\QQ.$ 
\end{defn}

Here, $USC(\QQ)$ and $LSC(\QQ)$ denote the space of upper and lower semicontinuous functions in $\QQ$, respectively. We refer to
\cite{Is97} for conditions that imply the (discontinuous)
comparison principle. For instance, if $H$ is convex, Equation~\eqref{HJ} satisfies the (discontinuous) comparison principle. This applies in particular to $H(\xi) = |\xi|^p$ for $p>1$. Concerning the case $H(\xi)=|\xi|^p$ for $p\in (0,1)$, Equation~\eqref{HJ} satisfies
the comparison principle in $C_b(\mathbb{R}^N)$ as recalled in
Appendix~\ref{apC} \cite{psbarles}. 

While the comparison principle seems to be the most effective tool in the study of uniqueness (for equations of the type considered here) we mention the proof in \cite{St02} concerning the uniqueness of the solution obtained by the Lax-Hopf formula.
 
 %%% ----------------------------------------------------------------------

\section{\textbf {Notation}}

Throughout the paper, we shall make use of the following standard functional notation.

The space $C^{2,1}(\mathbb{R}^N\times(0,\infty))$ is the space of all functions $u=u(x,t) $ which are twice continuously differentiable with respect to  $x\in\mathbb{R}^N$ and once with respect to $t\in (0,\infty)$.

The space $C^2_b(\mathbb{R}^N)$  is the space of all twice continuously differentiable functions $f$ such that all their derivatives up to second order are bounded (i.e., in $C_b(\mathbb{R}^N)$).

The space $W^{1,\infty}(\mathbb{R}^N)$ is the space of functions having uniformly bounded (distributional)  first order derivatives (i.e., using Rademacher's theorem, uniformly Lipschitz continuous functions).

The norm in $L^q(\mathbb{R}^N)$ is denoted by $\|\cdot\|_q, \quad q \in [1, \infty]$.

%%% ----------------------------------------------------------------------

\section{{\textbf {Results}}}

The existence and uniqueness results for solutions  to ~\eqref{HJV}-\eqref{init} are recalled in Proposition~\ref{basic} below. When the initial function is bounded, these solutions converge to a viscosity solution to \eqref{HJ}-\eqref{initial}, as expressed in the following theorem.

\begin{thm} \label{bounded}
Let $p \in (0,\infty), \quad p \neq 1$, and let $ \varphi \in
C_b(\mathbb{R}^N)$. Assume that $H$ satisfies the hypotheses
\eqref{Hass} and \eqref{Hass3} and that Equation~\eqref{HJ} satisfies
the comparison principle in $C_b(\mathbb{R}^N)$
(cf. Definition~\ref{comparison}~(b)). Assume also that $H$ satisfies
the growth condition 
\begin{equation}
 \label{Hass4}
\quad H(r)\leq\widetilde{H}(r):= g_H\
(r^{\kappa_\infty}+r^{\kappa_0}),\quad 0<r<\infty,  \;\;
0<\kappa_\infty\leq\kappa_0, 
\end{equation}
where $g_H>0 $ is a constant.
Then the solutions $\veps$ to  (\ref{HJV})-(\ref{init}) converge as $\varepsilon\to 0$ towards the unique (viscosity) solution $v$ to (\ref{HJ})-(\ref{initial}), uniformly in every compact subset of $\QQ$.
The function $v$ is differentiable a.e. in $\QQ$ and satisfies
(\ref{HJ}) at any  point of differentiability. Furthermore, we have,
for a.e. $(x,t) \in \QQ$, 
 \begin{equation}
|\nabla_x v(x,t)| \leq \lambda_p \|\varphi\|_{\infty}^{\frac{1}{p}} \,
(at)^{-\frac{1}{p}}  \, ,\label{vgradx} 
\end{equation}
\begin{equation}
\label{vdt}
0 \geq \partial_t v(x,t) \geq -L\ t^{-\mu},
\end{equation}
where $\lambda_p=1$ if $p>1$ and $\lambda_p=(2/p)^{\frac{1}{p}}$ if $p<1$, 
$$
\mu=\left\{
\begin{array}{lcl}
\displaystyle{\frac{\kappa_0}{p}} & \quad \text{if} \quad & 0<t\leq 1,\\
& & \\
\displaystyle{\frac{\kappa_\infty}{p}} & \quad \text{if}\quad & 1<t<\infty,
\end{array}
\right.
$$
and
$$
L=g_H\ \left\{ \left( 2 \lambda_p^p \|\varphi\|_{\infty} a^{-1}
\right)^{\frac{\kappa_0}{p}} + \left( 2 \lambda_p^p
\|\varphi\|_{\infty} a^{-1} \right)^{\frac{\kappa_\infty}{p}}
\right\}. 
$$
\end{thm}

As our aim in this paper is to derive estimates for the solutions $\veps$ to \eqref{HJV}-\eqref{init} (almost) independent of $\varepsilon$, the estimate \eqref{vdt} is obtained by passing to the limit as $\varepsilon\to 0$ in an analogous estimate for $\veps$ (see Proposition~\ref{bernstein} below). However, an alternative and simpler proof (with a slightly better constant than $L$) relies on \eqref{Hass4}, \eqref{vgradx}, and the fact that $v$ solves \eqref{HJ}-\eqref{initial} almost everywhere. Indeed, we infer from these properties that
\begin{eqnarray*}
\partial_t v & = & - H(|\nabla_x v|) \ge - g_H\ \left( |\nabla_x v|^{\kappa_\infty} + |\nabla_x v|^{\kappa_0} \right) \\
& \ge & - g_H\ \left\{ \left( \lambda_p^p \|\varphi\|_{\infty} a^{-1} \right)^{\frac{\kappa_0}{p}} + \left( \lambda_p^p
\|\varphi\|_{\infty} a^{-1} \right)^{\frac{\kappa_\infty}{p}} \right\}\ t^{-\mu} ,
\end{eqnarray*}
the parameter $\mu$ being defined in Theorem~\ref{bounded}. A further comment in that direction is that the vanishing viscosity approach used here (and already used in \cite{Li85}) is not the only route towards gradient or time derivative estimates, see, e.g., \cite{Ba90,Ba90b,Ka95,Li82}.

\begin{rem}\label{roptim}
If $H(r)=r^p$ for $p>1$, we have $\kappa_\infty=\kappa_0=p$ and
\eqref{vdt} indicates that $\partial_t v \ge -C/t$ for $t\ge 1$ for
some positive constant $C$ depending on $N$, $p$ and
$\|\varphi\|_\infty$. It gives a temporal decay rate for large times
of the same order as that obtained in \cite{BC81} where the inequality
$\partial_t v \ge -v/((p-1)t)$ (in the sense of distributions) is
established by using the homogeneity of the Hamiltonian $H$.  
\end{rem}

We now turn to the case where the initial function $\varphi$ is
continuous but not necessarily bounded. Thus, in contrast to the
previous theorem , where the positivity of $\varphi$ was not essential
(as $\varphi$ could be replaced by $\varphi+c$), the positivity
assumption (or rather the requirement that $\varphi$ be bounded from
below) in the following theorem is essential. Also, we need to impose
an additional growth assumption on $H$, namely that $H$ fulfills the
$p$-condition \eqref{Hass3} with $p>1$. 

\begin{thm}
\label {general}
Let $0 \leq \varphi \in C(\mathbb{R}^N)$ and assume  that $H$
satisfies the hypotheses~\eqref{Hass} and~\eqref{Hass3}  with $p>1$,
together with \eqref{Hass4}. Assume also that Equation~\eqref{HJ}
satisfies the (discontinuous) comparison principle
(cf. Definition~\ref{comparison}~(a)). Then the solutions $\veps$ to  (\ref{HJV})-(\ref{init}) converge as $\varepsilon\to 0$ towards the unique (viscosity) solution $v$ to (\ref{HJ})-(\ref{initial}), uniformly in every compact subset of $\QQ$. The function $v$ belongs to $W^{1,\infty}_{loc}(\QQ)$ and satisfies (\ref{HJ}) as well as \eqref{vgradx}, \eqref{vdt} for a.e. $(x,t) \in \QQ$. 
\end{thm}

\begin{rem}
\label{powersum}
When dealing with the Hamilton-Jacobi equation ~\eqref{HJ} it seems
unavoidable (as is the case in the references cited in this paper) to
have a rather long list of assumptions on the Hamiltonian
$H$. Furthermore, some results depend only on partial lists of the
assumptions, in addition to the interplay between the degree of
generality assumed on the initial data (and solutions) and the
corresponding assumptions. We therefore emphasize that
Theorem~\ref{bounded} is applicable to the case of sums of powers as
in ~\eqref{specialpl}, while Theorem ~\ref{general} is applicable to
the case ~\eqref{specialpl} when $\set{p_1,...,p_m}\in (1,\infty)^m.$ 
\end{rem}

\begin{rem}
Using a condition similar to our assumption ~\eqref{Hass3} Lions
\cite[Section~IV]{Li85} obtains viscosity solutions for bounded,
lower semicontinuous initial data. 
\end{rem}

%%% ----------------------------------------------------------------------

\section {\textbf{The viscous Hamilton-Jacobi equation}}
    
We first draw useful consequences of \eqref{Hass3} and \eqref{Hass4} on $H$. 

\begin{lem}
\label{lepropH}
Assume that $H$ fulfills \eqref{Hass} and \eqref{Hass3}. Then
\begin{equation}
\label{Hass5}
H(r)\geq \frac{a}{|p-1|}\ r^p,\quad r\geq 0.
\end{equation}
If $H$ also satisfies \eqref{Hass4} then $\kappa_\infty \le p \le \kappa_0$.
\end{lem}

\begin{proof} Assume first that $p>1$ in \eqref{Hass3}. Then, if $r>0$, $\delta\in(0,r)$, $s\in (\delta,r)$ and $\eta>0$, we infer from \eqref{Hass3} that 
$$
\frac{d}{ds} \left( \frac{\Phim(s^2)}{s} \right) = \frac{\Thetam(s^2)}{s^2} \ge a\ s^{p-2} - b \eta^\gamma\ s^{-2}.
$$
Integrating over $(\delta,r)$ with respect to $s$ gives
$$
\frac{\Phim(r^2)}{r} \ge \frac{\Phim(\delta^2)}{\delta} + \frac{a}{p-1}\ \left( r^{p-1} - \delta^{p-1} \right) + b \eta^\gamma\ \left( \frac{1}{r} - \frac{1}{\delta} \right) ,
$$
whence, after letting $\eta\to 0$ and using the nonnegativity of $H$, 
$$
\frac{H(r)}{r} \ge \frac{H(\delta)}{\delta} + \frac{a}{p-1}\ \left( r^{p-1} - \delta^{p-1} \right) \ge \frac{a}{p-1}\ \left( r^{p-1} - \delta^{p-1} \right) .
$$
Passing to the limit as $\delta\to 0$ gives \eqref{Hass5}. The proof is similar for $p\in (0,1)$ except than one integrates over $(r,A)$ for $A>r$ and then let $A\to\infty$.

Next, if $H$ also satisfies \eqref{Hass4}, the claimed constraints on $\kappa_0$ and $\kappa_\infty$ readily follow from \eqref{Hass5} by looking at the behavior for small $r$ and large $r$.
\end{proof}

We next recall  the basic existence and uniqueness theorem for regular initial data \cite{ABA98}. In fact, the result there refers to the special case \eqref{special}. However  the same method of proof can be used in order to obtain the following proposition \cite{adi}.

  \begin{prop}
  \label{basic}
  Let $0 \leq \varphi \in C^2_b(\mathbb{R}^N)$ and $H$ satisfy \eqref{Hass}. Then the Cauchy problem (\ref{HJV})-(\ref{init}) has a unique global solution  $\veps$ such that \newline

\begin{enumerate}
\item $\veps \in C^{2,1}(\mathbb{R}^N\times[0,\infty))$,\newline

\item $0 \leq \veps(x,t) \leq \|\varphi\|_{\infty}, \qquad (x,t) \in
\mathbb{R}^N\times(0,\infty)$, \newline

\item $\|\nabla_x \veps(\cdot,t)\|_{\infty} \leq \|\nabla_x \varphi\|_{\infty}\quad , \quad t>0.$
\end{enumerate}
\end{prop}

Observe that the assumption $0 \leq \varphi$ entails no loss of generality as $\varphi$ can be replaced by $ \varphi + \|\varphi\|_{\infty}$ without changing the equation. Proposition~\ref{basic} is actually valid assuming only the first property in \eqref{Hass}.

A remarkable fact (which is crucial in our study) concerning the solution $\veps$ is that its gradient can be estimated
\textit{independently} of $\varepsilon$ while only a mild dependence on $\varepsilon$ shows up for its time derivative. Such estimates are obtained by the ``Bernstein method'' \cite{Be81}, namely, using the comparison principle for a certain function of $\nabla_x \veps$ or $\partial_t \veps$. The estimates needed in this paper are gathered in the following proposition.

\begin{prop}
\label{bernstein}
Let $p \neq 1$ and $0 \leq \varphi \in C^2_b(\mathbb{R}^N)$. Assume that $H$ satisfies the hypotheses \eqref{Hass} and \eqref{Hass3}. Then the solution $\veps$ to  (\ref{HJV})-(\ref{init}) satisfies, with $p$ as in \eqref{Hass3},
\begin{equation}
\label{gradx}
\|\nabla_x \veps(\cdot,t)\|_{\infty} \leq \lambda_p \|\varphi\|_{\infty}^{\frac{1}{p}} \, (at)^{-\frac{1}{p}} \quad , \qquad t>0,
\end{equation}
where $\lambda_p=1$ if $p>1$ and $\lambda_p=(2/p)^{\frac{1}{p}}$ if $p<1.$\newline

In addition, if $p>1$,
\begin{equation}
\label{gradxind}
\left\|\nabla_x \left(\veps ^{\frac{p-1}{p}}\right)(\cdot,t)\right\|_{\infty} \leq \mu_p t^{-\frac{1}{p}} \quad , \qquad t>0,
\end{equation}
where $\mu_p=(p-1)a^{-\frac{1}{p}}/p.$

For the time derivatives, the following estimates hold. First, for $0<\varepsilon<\varepsilon_0,$ where $\varepsilon_0$ depends only on $p$ and $N$,
\begin{equation}
\label{dudtpl}
\partial_t \veps(x,t) \leq 2^{\frac{p+1}{p}}\ N\ \lambda_p\ \|\varphi\|_{\infty}^{\frac{1}{p}}\ a^{-\frac{1}{p}}\ \varepsilon^{\frac{1}{2}}\ t^{-\frac{p+2}{2p}}
\end{equation}
for $(x,t) \in \QQ$.

Next, assume in addition that \eqref{Hass4} is satisfied. Then, for all $0<\varepsilon<\varepsilon_0,$
\begin{equation}
\label{dudtmn}
\partial_t \veps(x,t) \geq - L\ t^{-\mu} - 2^{\frac{p+1}{p}}\ N\ \lambda_p\ \|\varphi\|_{\infty}^{\frac{1}{p}}\ a^{-\frac{1}{p}}\ \varepsilon^{\frac{1}{2}}\ t^{-\frac{p+2}{2p}}
\end{equation}
for $(x,t)\in\QQ$, the constants $\mu$ and $L$ being defined in Theorem~\ref{bounded}. 
\end{prop}

An estimate similar to \eqref{gradx} was obtained by Lions \cite[Section~I]{Li85} but with a dependence upon $\varepsilon$ which vanishes in the limit $\varepsilon\to 0$. The estimate \eqref{gradxind} (for the special case \eqref{special}) was first derived in  \cite{BL99}, and we follow it rather closely. Let us emphasize here that it is not only independent on $\varepsilon$ but also on the initial data, a property we shall use in the proof of Theorem~\ref{general}. The estimate \eqref{gradx} for the case $p<1$ (again for the special form \eqref{special}) was first derived in \cite{GGK03}.  However our proof   seems to be simpler. The estimates derived here for the time derivatives  generalize estimates obtained in \cite{Gi05} for the special case \eqref{special}. In this latter case we have $\kappa_0=\kappa_\infty=p,$ hence $\mu=1.$  

\begin{rem}\label{equit}
Remark that the estimates for the time derivative of $\veps$ are much
more complicated (and depend explicitly on the behavior of $H$ at
$r=0$ and $r=\infty$ as reflected in the additional assumption
\eqref{Hass4}). Such estimates are needed in order to ensure the
convergence of the solution $\veps$ to \eqref{HJV}-\eqref{init} (as
$\varepsilon\rightarrow 0$) in $W^{1,\infty}_{loc}(\QQ)$ to a
viscosity solution of ~\eqref{HJ}-\eqref{initial}. In fact, in
Appendix~\ref{apB} (see also Remark~\ref{simplet}) we show that the equicontinuity in $t$ of the family $\big\{\veps\big\}_{\varepsilon>0}$ can be obtained without the additional requirement \eqref{Hass4}. It follows, in view of the stability result for viscosity solutions (see, e.g.,
\cite[Th\'eor\`eme~2.3]{Ba94} or \cite[Theorem~1.4]{CEL84}) that the
limit function is a viscosity solution to \eqref{HJ}-\eqref{initial}. However, without the uniform boundedness in $W^{1,\infty}_{loc}(\QQ)$ it is not possible to show that the limiting solution  is differentiable a.e. and hence satisfies \eqref{HJ} a.e. (see \cite[Section~10.1]{Ev98}). 
\end{rem}

\begin{proof}[Proof of Proposition~\ref{bernstein}]
Observe that we can assume without loss of generality that $\veps>c$, where $c >0$ is arbitrary, by adding $c$ to $\varphi$ , provided the estimates do not depend on $c$.

We take the regularized function $\Phim$ as in  ~\eqref{Hass} and consider  the solution $\veta$ to the modified equation
\begin{eqnarray}
\partial_t \veta -\varepsilon \Delta \veta+\Phim(|\nabla \veta|^2) & = & 0, \qquad (x,t) \in \QQ, \label{HJV-mod}\\
\veta(x,0) & = & \varphi(x),\qquad  x\in \mathbb{R}^N
\label{initm}
\end{eqnarray}
for $\varepsilon>0$ and $\eta>0$. Inspecting the proof in \cite{ABA98} we see that the solution $\veta$ exists globally , belongs to $C^2_b(\mathbb{R}^N)$ for all $t \geq 0$ and satisfies $\veta>c$. In addition, it is smooth for $t>0$ and we can differentiate it as many times as needed. Also, as $\eta \rightarrow 0$,
$$
\left( \veta , \nabla\veta \right) \longrightarrow \left( \veps , \nabla\veps \right) \qquad  \text{uniformly in compact subsets of $\QQ$. }
$$

It therefore suffices to prove the estimates in Proposition~\ref{bernstein} for $\veta$, provided these estimates are independent of the positive constants $\eta$ and $c$. In what follows we simplify the notation by referring to $\veta$ as $V$.

We now consider a strictly monotone smooth function $f$ and define the functions $u$ and $w$ by
\begin{equation}
\label{w}
u=f^{-1}(V) \quad \mbox{ and } \quad w=|\nabla u|^2.
\end{equation}
Clearly $u$ and $w$ both belong to $C^{\infty}(\QQ)$ and their first derivatives in $x$ and $t$ are uniformly bounded and continuous in $\mathbb{R}^N \times [0,\infty)$. From (\ref{HJV-mod}) we obtain that $u$ solves
$$
f'(u) \left\{ \partial_t u- \varepsilon \Delta u -\varepsilon \frac{f''(u)}{f'(u)} |\nabla u|^2 +\frac{1}{f'(u)} \Phim(f'(u)^2 |\nabla u|^2 ) \right\} =0.
$$
so that
\begin{eqnarray}
\label{wt}
& & \partial_t w = 2 \nabla u \cdot \nabla \partial_t u \\ 
\nonumber
& = & 2 \nabla u \cdot \left\{ \varepsilon \left[ \Delta (\nabla u) +\nabla \left( \frac{f''(u)}{f'(u)}w \right) \right] -\nabla\left[ \frac{1}{f'(u)} \Phim(f'(u)^2 w) \right] \right\}\\ 
\nonumber & = & 2\varepsilon \nabla u\cdot\left[ \Delta(\nabla u) +\frac{f''(u)}{f'(u)}\nabla w \right]+2\varepsilon \left( \frac{f''}{f'}\right)'(u)w^2\\ 
\nonumber & + & 2\frac{f''(u)}{f'(u)^2}\Phim(f'(u)^2 w)w-2\Phim'(f'(u)^2 w)\ \left[ f'(u)\nabla u\cdot\nabla w +2 f''(u)w^2 \right].
\end{eqnarray}
Define the operator
$$
\mathcal{L}z=z_t-\varepsilon\Delta z+2\Phim'(f'(u)^2 |\nabla u|^2)f'(u)\nabla u\cdot\nabla z-2\varepsilon\frac{f''(u)}{f'(u)}\nabla u\cdot\nabla z.
$$
Noting that
\begin{equation}
\label{lapw}
\Delta w = 2\ \sum_{j=1}^N \sum_{k=1}^N  \left| \partial_{x_j}  \partial_{x_k} u\right|^2 + 2\nabla u\cdot \nabla(\Delta u) \ge 2\nabla u\cdot \nabla(\Delta u),
\end{equation}
we deduce from ~\eqref{wt} that
\begin{eqnarray}
\label{l2} 
\mathcal{L}w & \leq & 2\varepsilon\left( \frac{f''}{f'}\right)'(u)\ w^2 + 2 \frac{f''}{(f')^2}(u) \Phim(f'(u)^2 w)w\\ 
\nonumber
& &\ -4 \Phim'(f'(u)^2 w) f''(u) w^2\\ 
\nonumber& = & 2\varepsilon\left( \frac{f''}{f'}\right)'(u)\ w^2 - 2\frac{f''}{(f')^2}(u)\ \Thetam(f'(u)^2 w)w,
\end{eqnarray}
where $\Thetam$ is defined in Definition~\ref{pcondi}.

\medskip

We now specify the function $f$ and begin with the case $p>1$. We choose 
\begin{equation}
\label{f}
f(r)= r^{\frac{p}{p-1}},  \qquad r \geq 0,
\end{equation}
so that 
$$
-2\left( \frac{f''}{f'}\right)'(r)=\frac{2}{p-1}r^{-2}\geq 0 \quad \mbox{ and } \frac{f''}{(f')^2}(r)=\frac{1}{p}r^{-\frac{p}{p-1}}.
$$
Inserting these estimates in ~\eqref{l2} we get
$$
\mathcal{L}w + \frac{2}{p} u^{-\frac{p}{p-1}}\Thetam(f'(u)^2 w)w\leq 0.
$$
Owing to Definition~\ref{pcondi}~(i) we further obtain
$$\mathcal{L}w+\frac{2}{p}u^{-\frac{p}{p-1}} \left[ af'(u)^p w^{\frac{p}{2}}-b\eta^\gamma \right]w\leq 0,$$
and finally, inserting $f'(r)=p r^{\frac{1}{p-1}}/(p-1)$, we obtain
\begin{equation}
\label{lww}
\mathcal{L}w+\frac{2}{p-1}\left( \frac{p}{p-1} \right)^{p-1}aw^{1+\frac{p}{2}}\leq \frac{2b}{p}\eta^\gamma u^{-\frac{p}{p-1}}w.
\end{equation}
Note that $u^{-\frac{p}{p-1}}=V^{-1}\leq c^{-1}$, so that ~\eqref{lww}
yields, if we take $2b\eta^{\frac{\gamma}{2}}<pc,$
\begin{equation}
\label{lwww}
\mathcal{L}w+\frac{2a}{p-1}\left( \frac{p} {p-1} \right)^{p-1}w^{1+\frac{p}{2}}\leq \eta^{\frac{\gamma}{2}}w.
\end{equation}
Now consider the function $h_\eta(t)=K_\eta t^{-\frac{2}{p}}$, where
$K_\eta>0$ is a constant to be determined. We require $h_\eta$ to be a supersolution (in some time interval) to ~\eqref{lwww}, namely,
$$
\mathcal{L}h_\eta+\frac{2a}{p-1}\left( \frac{p} {p-1} \right)^{p-1} h_\eta^{1+\frac{p}{2}}\geq\eta^{\frac{\gamma}{2}}h_\eta.
$$
This condition is satisfied in $0<t<\eta^{-\frac{\gamma}{4}},$ if
$$-\frac{2}{p}K_\eta+\frac{2a}{p-1} \left( \frac{p}{p-1} \right)^{p-1} K_\eta^{1+\frac{p}{2}}=\eta^{\frac{\gamma}{4}}K_\eta,$$
or
\begin{equation}
K_\eta=\left( \frac{1}{a}+\frac{p}{2a}\eta^{\frac{\gamma}{4}}\right)^{\frac{2}{p}}\ \left( \frac{p-1}{p} \right)^2.
\end{equation}
The comparison principle now implies that $w(x,t)\leq h_\eta(t)$ for $x\in\mathbb{R}^N$ and $0<t<\eta^{-\frac{\gamma}{4}}$. Consequently, 
\begin{equation}
\label{lwwww}
\left\| \nabla \left( (\veta)^{\frac{p-1}{p}} \right)(.,t) \right\|_\infty\leq K_\eta^{\frac{1}{2}}  t^{-\frac{1}{p}}, \qquad 0<t<\eta^{-\frac{\gamma}{4}}, \quad 0<\frac{2b}{p}\eta^{\frac{\gamma}{2}}<c.
\end{equation}
In addition, combining \eqref{lwwww} with the bound $\veta\le\|\varphi\|_\infty + c$ gives
\begin{equation}
\label{volvic}
\|\nabla\veta(.,t)\|_\infty\leq \frac{p K_\eta^{\frac{1}{2}}}{p-1}\ \left( \|\varphi\|_\infty + c \right)^{\frac{1}{p}}  t^{-\frac{1}{p}}, \qquad 0<t<\eta^{-\frac{\gamma}{4}}, \quad 0<\frac{2b}{p}\eta^{\frac{\gamma}{2}}<c.
\end{equation}

These estimates are independent of $\varepsilon>0$, and by letting $\eta \rightarrow 0$, and then $c\rightarrow 0$, we obtain \eqref{gradxind} and \eqref{gradx} in the case $p>1$.

\medskip

We now turn to the case $0<p<1.$ Our starting point is again the inequality ~\eqref{l2} with the same function $\Phim$ but with a different choice of the function $f$. More precisely, instead of \eqref{f}, we take
\begin{equation}
\label{f1}
f(r)=2\eta^{\frac{\gamma}{2}}+\|\varphi\|_{\infty}-\frac{1}{2}r^2.
\end{equation}
Then
$$
-2\left( \frac{f''}{f'}\right)'(r)=\frac{2}{r^2}\geq 0 \quad \mbox{ and } \frac{f''}{(f')^2}(r)=-\frac{1}{r^2}.
$$
Inserting these estimates in \eqref{l2} we get
$$
\mathcal{L}w - \frac{2}{u^2}\ \Thetam(f'(u)^2 w)w\leq 0,
$$
so that in conjunction with Definition~\ref{pcondi}~(ii) we obtain
$$
\mathcal{L}w+2au^{p-2}w^{1+\frac{p}{2}}\leq\frac{2b \eta^\gamma}{u^2}\ w.
$$
Taking $0<c<\eta^{\frac{\gamma}{2}}$ the maximum principle (for $V=\veta$) implies that
\begin{equation}
\label{uestim}
\eta^{\frac{\gamma}{2}}\leq \frac{1}{2}u^2\leq 2\eta^{\frac{\gamma}{2}}+\|\varphi\|_{\infty}. 
\end{equation}
This estimate provides an upper bound for the right-hand side of the above inequality and leads us to
\begin{equation}
\label{Lw}
\mathcal{L}w+2au^{p-2}w^{1+\frac{p}{2}}\leq b \eta^{\frac{\gamma}{2}} w.
\end{equation}
Also by \eqref{uestim}, since $p-2<0,$
$$
u^{p-2}\geq \left[ 2(2\eta^{\frac{\gamma}{2}}+\|\varphi\|_{\infty})\right]^{\frac{p-2}{2}},
$$
so that ~\eqref{Lw} yields
\begin{equation}
\label{Lww}
\mathcal{L}w+2^{\frac{p}{2}}a (2\eta^{\frac{\gamma}{2}}+\|\varphi\|_{\infty})^{\frac{p-2}{2}}w^{1+\frac{p}{2}} \leq b \eta^{\frac{\gamma}{2}} w.
\end{equation}
As above, we now try a supersolution to \eqref{Lww} of the form $h_\eta(t)=K_\eta t^{-\frac{2}{p}}$ (in a certain time interval). We therefore need
$$
-\frac{2}{p}+2^{\frac{p}{2}}a (2\eta^{\frac{\gamma}{2}}+\|\varphi\|_{\infty})^{\frac{p-2}{2}}K_\eta^{\frac{p}{2}} \geq b \eta^{\frac{\gamma}{2}} t,
$$
hence for $0<t<\eta^{-\frac{\gamma}{4}}$ we can take
$$
K_\eta=\left( \frac{2+bp\eta^{\frac{\gamma}{4}}}{2^{\frac{p}{2}}ap}\right)^{\frac{2}{p}}\ \left( 2\eta^{\frac{\gamma}{2}}+\|\varphi\|_{\infty} \right)^{\frac{2-p}{p}}.
$$
The comparison principle then entails that $w(x,t)\le K_\eta t^{-\frac{2}{p}}$ for $x\in\mathbb{R}^N$ and $0<t<\eta^{-\frac{\gamma}{4}}$. 
Using ~\eqref{f1} and ~\eqref{uestim} we conclude that
$$
\|\nabla\veta(\cdot,t)\|_{\infty} \leq \| u(\cdot,t)\|_{\infty} \|\nabla u(\cdot,t)\|_{\infty} \leq \left( \frac{2+bp\eta^{\frac{\gamma}{4}}}{pa} \right)^{\frac{1}{p}}\ (2\eta^{\frac{p\gamma}{2}}+\|\varphi\|_{\infty})^{\frac{1}{p}} t^{-\frac{1}{p}}
$$
for $0<t<\eta^{-\frac{\gamma}{4}}$ and $0<c<\eta^{\frac{\gamma}{2}}$.
This estimate is independent of $\varepsilon>0.$ Letting $c\rightarrow 0$ and then $\eta\rightarrow 0$ we obtain ~\eqref{gradx} for $0<p<1$
with $\lambda_p= \big( 2/p\big)^{\frac{1}{p}}.$

\bigskip

We next turn to the proof of \eqref{dudtpl} and~\eqref{dudtmn}. We still work with the modified equation \eqref{HJV-mod} and simplify as before the notation by setting $\veta=V.$ We follow the idea of proof in \cite[Lemma 10]{Gi05}.

Let $M>0$ and $\vartheta>0$ be positive constants (to be specified later) and define $\Gamma=M + N\ \left( \|\nabla\varphi\|_\infty^2-|\nabla V|^2 \right)/(4M)$ and $w=\left( \delta\ \partial_t V - \vartheta \right)/\Gamma$ for $\delta\in\{-1,1\}$. From \eqref{HJV-mod} we get readily
$$
\partial_t w = -\frac{w}{\Gamma} \partial_t \Gamma+\frac{\varepsilon\Delta(\Gamma w)}{\Gamma}- \frac{2}{\Gamma}\ \Phim'(|\nabla V|^2)\nabla V\cdot\nabla(\Gamma w)
$$
which we can rewrite as
\begin{equation}
\label{lin-w}
\partial_t w = \frac{N}{4M}\ \frac{w}{\Gamma}\ A+B\cdot\nabla w+\varepsilon \Delta w,
\end{equation}
where 
$$
B = 2 \varepsilon \frac{\nabla\Gamma}{\Gamma} - 2 \Phim'\left( |\nabla V|^2 \right) \nabla V
$$ 
is a bounded continuous function and
$$
A= \partial_t\left( |\nabla V|^2 \right) - \varepsilon\Delta\left( |\nabla V|^2 \right) + 2 \Phim'\left( |\nabla V|^2 \right) \nabla V\cdot \nabla\left( |\nabla V|^2 \right) .
$$
Recalling that (cf. \eqref{lapw})
$$
\Delta\left( |\nabla V|^2 \right) =  2\nabla V\cdot \nabla(\Delta V)
+ 2 \sum\limits_{j=1}^N\sum\limits_{k=1}^N(\partial_{x_j}\partial_{x_k}V)^2
$$
and
$$
\partial_t \left( |\nabla V|^2 \right) - 2\varepsilon\nabla V\cdot \nabla(\Delta V) + 2 \Phim'\left( |\nabla V|^2 \right) \nabla V\cdot \nabla\left( |\nabla V|^2 \right)=0
$$
by~\eqref{wt} (with $f(r)=r$ so that $u=V$), we obtain
$$
A = 2\varepsilon \nabla V\cdot \nabla(\Delta V)- \varepsilon \Delta\left( |\nabla V|^2 \right) = -2\ \varepsilon \sum\limits_{j=1}^N \sum\limits_{k=1}^N (\partial_{x_j}\partial_{x_k}V)^2.
$$
Then \eqref{lin-w} reads
$$
\partial_t w - \varepsilon \Delta w - B\cdot\nabla w + \frac{N \varepsilon}{2M}\ \frac{w}{\Gamma}\ \left( \sum\limits_{j=1}^N \sum\limits_{k=1}^N (\partial_{x_j}\partial_{x_k}V)^2 \right)=0
$$
or
$$
\partial_t w - \varepsilon \Delta w - B\cdot\nabla w + \frac{\varepsilon}{2M}\ A_1\  \frac{w}{\Gamma} + \frac{\varepsilon}{2M}\ \frac{w}{\Gamma}\  |\Delta V|^2  = 0
$$
with
$$
A_1 = N\ \sum\limits_{j=1}^N \sum\limits_{k=1}^N (\partial_{x_j}\partial_{x_k}V)^2 - |\Delta V|^2 \ge 0
$$
by the Cauchy-Schwarz inequality. Noting that 
$$
 |\Delta V|^2 = \frac{1}{\varepsilon^2}\ \left( w \Gamma + \vartheta + \delta\ \Phim\left( |\nabla V|^2 \right) \right)^2
$$
by \eqref{HJV-mod} and introducing the differential operator
\begin{eqnarray*}
\mathcal{M} z & = & \partial_t z - \varepsilon \Delta z - B\cdot\nabla
z + \frac{\Gamma}{2M\varepsilon}\ z^3 + \frac{1}{M\varepsilon}\ \left(
\vartheta + \delta\ \Phim\left( |\nabla V|^2 \right) \right)\ z^2 \\ 
& & +\ \frac{\varepsilon}{2M\Gamma}\ \left\{ A_1 + \frac{\left(
\vartheta + \delta\ \Phim\left( |\nabla V|^2 \right)
\right)^2}{\varepsilon^2} \right\}\ z , 
\end{eqnarray*}
we realize that $w$ solves
\begin{equation}
\label{alet}
\mathcal{M} w = 0 \quad \mbox{ in } \quad \mathbb{R}^N\times (0,\infty).
\end{equation}

\medskip

We first take $\delta=-1$ and 
$$
\vartheta = \sup_{r\in [0,\|\nabla\varphi\|_\infty]}\{\Phim(r^2)\}
$$
in the definition of $w$. As $\Gamma\ge M$, we infer from the
nonnegativity of $A_1$, $\vartheta$, $\Phim$ and
Proposition~\ref{basic}~(3) that $W(t) =
(\varepsilon/t)^{\frac{1}{2}}$ satisfies 
$$
\mathcal{M} W \ge - \frac{\sqrt{\varepsilon}}{2}\ t^{-\frac{3}{2}} +
\frac{\Gamma}{2M\varepsilon}\ W^3(t) \ge  0 . 
$$
Therefore $W$ is a supersolution to (\ref{alet}) and the comparison
principle entails that $w\le W$ in $\mathbb{R}^N\times
(0,\infty)$. Consequently, 
$$
- \partial_t V(x,t) - \sup_{r\in
[0,\|\nabla\varphi\|_\infty]}\{\Phim(r^2)\}\le \left(
\frac{\varepsilon}{t} \right)^{\frac{1}{2}}\ \Gamma(x,t) \le \left( M
+ \frac{N}{4M}\ \|\nabla\varphi\|_\infty^2 \right)\ \left(
\frac{\varepsilon}{t} \right)^{\frac{1}{2}} . 
$$
Choosing $M=\|\nabla\varphi\|_\infty$, we end up with
\begin{equation}
\label{baffe}
\partial_t V(x,t) \ge - \sup_{r\in
[0,\|\nabla\varphi\|_\infty]}\{\Phim(r^2)\} - \frac{N+4}{4}\
\|\nabla\varphi\|_\infty\ \left( \frac{\varepsilon}{t}
\right)^{\frac{1}{2}} . 
\end{equation}

\medskip

We next take $\delta=1$ and $\vartheta=0$ in the definition of $w$. As
above, it follows from the nonnegativity of $A_1$ and $\Phim$ and the
bound $\Gamma\ge M$ that the function $W$ satisfies $\mathcal{M} W\ge
0$ in $\mathbb{R}^N\times (0,\infty)$, whence $w\le W$ by the
comparison principle. Therefore,  
$$
\partial_t V(x,t) \le \left( M + \frac{N}{4M}\
\|\nabla\varphi\|_\infty^2 \right)\ \left( \frac{\varepsilon}{t}
\right)^{\frac{1}{2}} , 
$$
and the choice $M=\|\nabla\varphi\|_\infty$ gives
\begin{equation}
\label{torgnole}
\partial_t V(x,t) \le \frac{N+4}{4}\ \|\nabla\varphi\|_\infty\ \left(
\frac{\varepsilon}{t} \right)^{\frac{1}{2}} . 
\end{equation}

We then pass to the limit as $\eta\to 0$ and infer from (\ref{Hass3})
and the convergence of $\left( v_\varepsilon^\eta \right)_{\eta}$
towards $v_\varepsilon$ that 
$$
-\sup_{r\in [0,\|\nabla\varphi\|_\infty]}\{H(r)\} -\frac{N+4}{4}\
\|\nabla\varphi\|_\infty\ \left( \frac{\varepsilon}{t}
\right)^{\frac{1}{2}} \le \partial_t v_\varepsilon(x,t) \le
\frac{N+4}{4}\ \|\nabla\varphi\|_\infty\ \left( \frac{\varepsilon}{t}
\right)^{\frac{1}{2}} 
$$
for $(x,t)\in \mathbb{R}^N\times (0,\infty)$. We finally use
(\ref{Hass4}) to conclude that 
\begin{eqnarray*}
\partial_t v_\varepsilon(x,t) & \ge & -g_H\ \left(
\|\nabla\varphi\|_\infty^{\kappa_\infty} +
\|\nabla\varphi\|_\infty^{\kappa_0} \right)-\frac{N+4}{4}\
\|\nabla\varphi\|_\infty\ \left( \frac{\varepsilon}{t}
\right)^{\frac{1}{2}} \\ 
\partial_t v_\varepsilon(x,t) & \le & \frac{N+4}{4}\
\|\nabla\varphi\|_\infty\ \left( \frac{\varepsilon}{t}
\right)^{\frac{1}{2}} 
\end{eqnarray*}
for $(x,t)\in \mathbb{R}^N\times (0,\infty)$. But, since
\eqref{HJV-mod} is an autonomous equation, we also have 
\begin{eqnarray*}
\partial_t v_\varepsilon(x,t) & \ge & -g_H\ \left( \left\|\nabla
v_\varepsilon\left( \frac{t}{2} \right)\right\|_\infty^{\kappa_\infty}
+ \left\|\nabla v_\varepsilon\left( \frac{t}{2}
\right)\right\|_\infty^{\kappa_0} \right) \\ 
& & -\ \frac{N+4}{4}\ \left\|\nabla v_\varepsilon\left( \frac{t}{2}
\right)\right\|_\infty\ \left( \frac{2\varepsilon}{t}
\right)^{\frac{1}{2}} , \\ 
\partial_t v_\varepsilon(x,t) & \le & \frac{N+4}{4}\ \left\|\nabla
v_\varepsilon\left( \frac{t}{2} \right)\right\|_\infty\ \left(
\frac{2\varepsilon}{t} \right)^{\frac{1}{2}} 
\end{eqnarray*}
for $(x,t)\in \mathbb{R}^N\times (0,\infty)$. Inserting \eqref{gradx}
in the above estimates and using that $N+4\le 4\sqrt{2} N$ complete
the proof of \eqref{dudtpl} and \eqref{dudtmn}. 
\end{proof}

\medskip

As already mentioned, in the particular case where $H$ is given by \eqref{special}, we have $\kappa_0=\kappa_\infty=p,$ and thus $\mu=1.$ We can then derive a better estimate for the time derivative, using a scaling argument as follows.

\begin{cor}
\label{homog}
Let $H$ be of the special form \eqref{special}. Then for every $\rho>0$ there exists a constant $C>0,$ depending only on $p,N,\|\varphi\|_{\infty},\rho$ such that, for all $0<\varepsilon<\varepsilon_0$, 
\begin{equation}
\label{t-1}
|\partial_t \veps(x,t)|\leq C t^{-1},\qquad \begin{cases}
(x,t) \in \mathbb{R}^N\times(\rho,\infty),\quad p\in (0,1)\cup (1,2],\\
(x,t) \in \mathbb{R}^N\times(0,\rho),\quad p\in [2,\infty).
\end{cases}
\end{equation}
\end{cor}

\begin{proof}
Note that $\veps$ satisfies, in view of \eqref{dudtpl}-\eqref{dudtmn}, the estimate
$$
\|\partial_t \veps(\cdot,\rho)\|_\infty\leq c_0, \qquad
0<\varepsilon<\varepsilon_0,
$$
where $c_0$ is independent of $\varepsilon .$ Define the function
$$V(y,\tau)=\veps(r^\beta y,r^\alpha\tau),\qquad (y,\tau)\in\QQ, \quad r>0.$$
It satisfies the equation (using the special form of $H$)
$$
\partial_\tau V(y,\tau)+r^{\alpha-p\beta}|\nabla V(y,\tau)|^p=\varepsilon r^{\alpha-2\beta}\Delta V(y,\tau).
$$
Assume first that $0<p\leq 2$ and take $\alpha=1$, $\beta=p^{-1}$ and $r>1$. Then $\varepsilon r^{\alpha-2\beta}<\varepsilon<\varepsilon_0$ hence
$$\|\partial_\tau V(\cdot,\rho)\|_\infty\leq c_0, \qquad$$
and turning back to $\veps$ with $t=r\rho$ we obtain \eqref{t-1} for $0<p\leq 2$. In the case $p>2$ we repeat the same argument, but with $r<1$.
\end{proof}

In view of the fact that only $\|\varphi\|_{\infty}$ appears in the  estimates , we can follow the methodology of \cite{GGK03} and extend the result of Proposition~\ref{basic} as follows.

  \begin{cor}
  \label{general-cor}
 Let $0 \leq \varphi \in C_b(\mathbb{R}^N)$, and let $H$ satisfy the
 hypotheses ~\eqref{Hass} and ~\eqref{Hass3}. Then
 (\ref{HJV})-(\ref{init}) has a unique global solution
 $\veps$ such that 
\begin{itemize}
\item[(i)] $\veps \in C^{2,1}(\mathbb{R}^N\times(0,\infty))\cap
C(\mathbb{R}^N\times[0,\infty)),$  
\item[(ii)] $0 \leq \veps(x,t) \leq \|\varphi\|_{\infty}, \qquad (x,t) \in
 \mathbb{R}^N\times(0,\infty)$,
\item[(iii)] $\veps$ satisfies in $\mathbb{R}^N\times(0,\infty)$ all the
estimates of  Proposition~\ref{bernstein}, the estimate \eqref{dudtmn}
being only true if $H$ fulfills the additional assumption \eqref{Hass4}.
\end{itemize}

 In addition, if $\varphi \in W^{1,\infty}(\mathbb{R}^N)$, then
$$
  \|\nabla_x \veps(\cdot,t)\|_{\infty} \leq \|\nabla_x
\varphi\|_{\infty}\quad , \quad t>0. 
$$
 \end{cor}
 
\begin{rem}
\label{simplet}
In contrast to the rather involved proof of \eqref{dudtpl}-\eqref{dudtmn}, it is quite easy to show that
$\veps$ belongs to $C([0,\infty),L^\infty(\mathbb{R}^N))$ (and is in
fact  H\"{o}lder continuous with respect to time, see Proposition~\ref{pim} in Appendix~\ref{apB} below).  Such an estimate
is sufficient for proving the uniform convergence (in compact subsets)
of a subsequence $\{v_{\varepsilon_j}\}_{j\ge 1}$ (as
$\varepsilon_j\rightarrow 0$), when the estimate ~\eqref{gradx} is
known, using the Arzela-Ascoli theorem. It follows, in view of the
stability result for viscosity solutions \cite[Th\'eor\`eme~2.3]{Ba94} that the limit function is a viscosity solution to ~\eqref{HJ}-\eqref{initial}. 
\end{rem}

%%% ----------------------------------------------------------------------

\section {{\textbf{Proof of Theorem \ref{bounded}}}}

Consider $0\le\varphi\in C_b(\mathbb{R}^N)$ and let $\veps$ be the
solution to  (\ref{HJV})-(\ref{init}) given in Corollary
\ref{general-cor}. In view of Proposition~\ref{basic}~(2),
(\ref{gradx}) and (\ref{dudtpl})-(\ref{dudtmn}), the family
$\{\veps\}_{\varepsilon \in (0,\varepsilon_0)}$ is uniformly bounded
in $\QQ$ and also bounded in $W^{1,\infty}_{loc}(\QQ)$. It follows
that there exist a subsequence $\{v_{{\varepsilon}_j}\}$,
${\varepsilon}_j \to 0$ and a function $v \in C_b(\QQ)\cap W^{1,\infty}_{loc}(\QQ)$ such that
\begin{equation}
 \label{visc-conv}
v_{{\varepsilon}_j} \xrightarrow[\varepsilon_j\rightarrow {0}]{}v,  \qquad
\text{uniformly in every compact subset of} \quad \QQ. 
\end{equation} 
The differentiability (a.e. in $ \QQ$) and the inequalities  (\ref{vgradx}), (\ref{vdt}) now follow from Rademacher's theorem \cite[Chapter~5] {Ev98} and Proposition~\ref{bernstein}. 

The limit function $v$ satisfies (\ref{HJ}) a.e. in $\QQ$. Indeed, the
convergence \eqref{visc-conv} implies, as in \cite[Chapter~10]{Ev98},
that $v$ is a ``viscosity solution'' to \eqref{HJ} and therefore it
satisfies (\ref{HJ}) at any point where it is differentiable. 

Next, we need to show that $v$ attains the assigned initial condition
(\ref{initial}). In view of ~\eqref{HJ} and the nonnegativity of $H$
we have $\partial_t v\leq 0$ a.e. in $\QQ$, whence $v(x,t_1)\le
v(x,t_2)$ for $x\in\mathbb{R}^N$ and $t_2>t_1>0$ owing to the
continuity of $v$ in $\QQ$. Recalling that $0\le v_\varepsilon\le
\|\varphi\|_\infty$ by Proposition~\ref{basic}~(2), the function 
\begin{equation}
\label{spirou}
v_0(x) = \sup_{t>0} v(x,t) = \lim_{t\to 0} v(x,t) \in
[0,\|\varphi\|_\infty]         
\end{equation}
is thus well-defined and satisfies
$$
0 \le v_0(x) \le \|\varphi\|_\infty \quad\mbox{ for }\quad x\in\mathbb{R}^N.
$$

We now identify $v_0$. Assume first that $\varphi\in
C^{\infty}_0(\mathbb{R}^N)$ (it actually suffices to assume
$\varphi\in  W^{1,\infty}(\mathbb{R}^N))$ and consider $t>0$. Then,
multiplying ~\eqref{HJV} by any $\psi\in  C^{\infty}_0(\mathbb{R}^N)$
and integrating over $\mathbb{R}^N\times (0,t)$, we get 
\begin{eqnarray*}
\left|
\int_{\mathbb{R}^N}\veps(x,t)\psi(x)dx-\int_{\mathbb{R}^N}\varphi(x)\psi(x)dx
\right| & \leq & \varepsilon t\|\varphi\|_{\infty}\|\Delta\psi(x)\|_1
\\ 
& + & t\|\psi\|_1\max\limits_ {0\leq s\leq \|\nabla\varphi\|_{\infty}}\,H(s),
\end{eqnarray*}
where we have used the estimates in Proposition~\ref{basic}. Letting
 $\varepsilon=\varepsilon_j$ and $j\to\infty$ we obtain
$$
\left|
\int_{\mathbb{R}^N}v(x,t)\psi(x)dx-\int_{\mathbb{R}^N}\varphi(x)\psi(x)dx
\right| \leq t\|\psi\|_1\max\limits_ {0\leq s\leq
\|\nabla\varphi\|_{\infty}}\,H(s) 
$$
which yields, by taking $t\downarrow 0$
\begin{equation}
\label{smooth}
v_0(x)=\varphi(x), \mbox{if} \quad \varphi\in  C^{\infty}_0(\mathbb{R}^N).
\end{equation}

Coming back to the general case $\varphi\in  C_b(\mathbb{R}^N)$ we consider $0\leq\psi\in C^{\infty}_0(\mathbb{R}^N)$ and $t>0$. Multiplying ~\eqref{HJV} by $\psi$, integrating over $\mathbb{R}^N\times (0,t)$ and using the positivity of $H$ we get
$$
\int_{\mathbb{R}^N}\veps(x,t)\psi(x)dx\leq\int_{\mathbb{R}^N}\varphi(x)\psi(x)dx+\varepsilon\int_0^t\int_{\mathbb{R}^N}\veps(x,s)\Delta\psi(x)dxds,
$$
which yields, by taking the sequence $\varepsilon=\varepsilon_j$ and letting $j\to\infty$,
$$
\int_{\mathbb{R}^N}v(x,t)\psi(x)dx\leq\int_{\mathbb{R}^N}\varphi(x)\psi(x)dx.
$$
It follows that, by taking the limit $t\downarrow 0$
\begin{equation}
\label{less}
v_0(x)\leq \varphi(x),\qquad \mbox {for a.e.} \quad x \in\mathbb{R}^N.
\end{equation}

To prove the opposite inequality, we first observe that, if $\varphi(x_0)=0$ for some $x_0\in\mathbb{R}^N$, then $v_0(x_0) = 0$ by \eqref{less}. Next, let $x_0 \in\mathbb{R}^N$ be such that $\varphi(x_0)>0.$ For $\eta>0$ sufficiently small let
$$
B^{\eta}=B_{\delta(\eta)}(x_0)=\{x \in \mathbb{R}^N,\quad |x-x_0|<\delta(\eta)\}
$$
be a ball such that
$$\varphi(x)\geq (1-\eta)\varphi(x_0),\qquad x\in B^{\eta}.$$
Consider now $0\leq\psi_\eta\in C^{\infty}_0(B^\eta)$ such that
$\psi_\eta(x)\leq (1-\eta)\varphi(x_0)$ with equality at $x=x_0.$ Let
$\Psi_\eta$ denote the solution to ~\eqref{HJ} (constructed as above),
with initial condition $\psi_\eta.$ By the comparison principle for viscosity solutions we have $\Psi_\eta(x,t)\leq v(x,t)$ for $(x,t)\in\QQ$. 
However, in view of ~\eqref{smooth} it follows that
$$ 
\Psi_\eta(x,t)\xrightarrow[t\rightarrow {0}^{+}]{}\psi_\eta(x),\qquad x
\in\mathbb{R}^N
$$
so that by the previous inequality $\psi_\eta(x)\leq v_0(x)$  for a.e. $x  \in\mathbb{R}^N$. If now $x_0$ is a Lebesgue point of $v_0$, the last
 inequality implies that $(1-\eta)\varphi(x_0)=\psi_\eta(x_0)\leq v_0(x_0)$, and by sending $\eta$ to $0$ we get for such a point $\varphi(x_0)\leq v_0(x_0)$. Thus, finally
$$
\varphi(x)\leq v_0(x) \qquad \mbox{for a.e.} \quad x  \in\mathbb{R}^N.
$$
Combining this inequality with ~\eqref{less} we get $\varphi(x)= v_0(x)$ for $x  \in\mathbb{R}^N$. Finally, as $\varphi\in C(\mathbb{R}^N)$, the time monotonicity of $v$ and the Dini theorem warrant that
$$
v(x,t)\xrightarrow[t\rightarrow {0}^{+}]{}\varphi(x)\qquad \mbox{uniformly in compact subsets of }\quad \mathbb{R}^N.
$$

The uniqueness  of the solution  follows from the fact that Equation~\eqref{HJ} satisfies the comparison principle in $C_b(\mathbb{R}^N)$.

%%% ----------------------------------------------------------------------

\section{{\textbf{Proof of Theorem~\ref{general}}}}

We begin by noting that since $\varphi$ is only assumed to be continuous (but not necessarily bounded), we cannot invoke Corollary~\ref{general-cor} . The existence of a solution $\veps$ to (\ref{HJV})-(\ref{init}) is therefore not guaranteed and must be addressed as a first step towards the study of a ``vanishing viscosity solution''.

For any integer $n\geq 1$ we set
$$\varphi_n=\min\{\varphi,n\}\in C_b(\mathbb{R}^N),$$
and let $\vepsn$ be the solution to ~\eqref{HJV} subject to the initial
condition $\vepsn(x,0)=\varphi_n(x).$ In view of Corollary~\ref{general-cor}
$$\vepsn \in C^{2,1}(\mathbb{R}^N\times(0,\infty))\cap C(\mathbb{R}^N\times[0,\infty)),$$
$$0 \leq \vepsn(x,t) \leq \|\varphi_n\|_{\infty}, \qquad (x,t) \in \QQ,$$
and the estimate ~\eqref{gradxind} is satisfied. Moreover, by Theorem~\ref{bounded} we have for any fixed $n$,
$$
\vepsn \xrightarrow[\varepsilon\rightarrow {0}^{+}]{}v_{0,n},
\qquad \text{uniformly in every compact subset of} \quad \QQ,
$$
where the limit function $v_{0,n}$ is differentiable a.e. in $\QQ$ with $v_{0,n}(.,0)=\varphi_n$ and satisfies  (\ref{HJ}) at any  point of differentiability.

Next we  show that the family $\{\vepsn\}_{n\geq 1}$ is uniformly bounded in every compact subset of $\QQ.$ To this end we follow \cite{VW94} and state the following lemma.
\begin{lem}
\label{influence}
Let $z\in C^{2,1}(\QQ)$ be any classical solution to \eqref{HJV} for some $\varepsilon\in (0,1).$ Assume that $H$ satisfies the assumptions \eqref{Hass} and \eqref{Hass3} with $p>1$. Then, for any $t>s>0$ and all $y \in \mathbb{R}^N$ and $R>0$,
\begin{equation}
\label{est-ball}
\int_{B_R(y)}z(x,t)dx\leq \int_{B_{2R}(y)}z(x,s)dx+C(t-s)R^N(1+R^{-\frac{p}{p-1}}),
\end{equation}
where $B_r(y)=\{x \in \mathbb{R}^N,\quad |x-y|<r\},$ and $C>0$ depends only on $N$, $p$ and $a$ (but, in particular, not on $\varepsilon$).
\end{lem}

The proof of the lemma is given at the end of this section.

We now continue with the proof of Theorem~\ref{general}. In what follows we use $C>0$ to denote various constants depending only on $p$, $N$ and $a$ unless explicit dependence on other parameters is indicated. 

Let $t>0.$ In view of ~\eqref{gradxind} we have, for any $x,y \in \mathbb{R}^N,$
$$
\vepsn(y,t)^\frac{p-1}{p}\leq \vepsn(x,t)^\frac{p-1}{p}+\mu_pt^{-\frac{1}{p}}|x-y|,
$$
hence, since $p>1,$
$$
\vepsn(y,t)\leq C \left[ \vepsn(x,t)+t^{-\frac{1}{p-1}}|x-y|^\frac{p}{p-1} \right].
$$
Integrating this inequality over $B_R(y)$ with respect to $x$ we get
$$
\int_{B_R(y)}\vepsn(y,t)dx\leq C\left[ \int_{B_R(y)}\vepsn(x,t)dx+t^{-\frac{1}{p-1}}\int_{B_R(0)}|x|^\frac{p}{p-1}dx \right].
$$
We now invoke the estimate ~\eqref{est-ball} for $z=\vepsn$ and $s=0$ (which is possible since $\vepsn$ is continuous at $s=0$) in the
right-hand side of the last inequality.
\begin{eqnarray*}
\vepsn(y,t) & \leq & C R^{-N}\left[ \int_{B_{2R}(y)}\varphi_n(x)dx+tR^N \left( 1+R^{-\frac{p}{p-1}} \right)+ t^{-\frac{1}{p-1}}R^{N+\frac{p}{p-1}} \right]\\
& \leq & C \left[ R^{-N}\int_{B_{2R}(y)}\varphi(x)dx+t \left( 1+R^{-\frac{p}{p-1}} \right)+t^{-\frac{1}{p-1}}R^{\frac{p}{p-1}} \right].
\end{eqnarray*}
It follows that in any cylinder $Q=B_R(0)\times (t_1,t_2)$, $t_2>t_1>0,$
$$\|\vepsn\|_{L^\infty(Q)}\leq C(R,t_1,t_2,\varphi),\qquad n\ge 1 ,$$
and, in view of ~\eqref{gradxind},
$$\|\nabla_x\vepsn\|_{L^\infty(Q)}\leq \frac{p \mu_p}{p-1} \|\vepsn\|_{L^\infty(Q)}^{\frac{1}{p}}t_1^{-\frac{1}{p}}\leq C(R,t_1,t_2,\varphi) , \qquad n\ge 1.$$

In view of Theorem~\ref{bounded} we have, by passing to the limit
$\varepsilon\rightarrow 0$
\begin{equation}
\label{v0n}
\partial_t v_{0,n} + H(|\nabla_x v_{0,n}|)=0, \quad\mbox{for a.e.} \qquad (x,t) \in \QQ
\end{equation}
The last two estimates and ~\eqref{v0n}  yield
\begin{equation}
\label{Wloc}
\|v_{0,n}\|_{W^{1,\infty}(Q)}\leq C(R,t_1,t_2,\varphi),\qquad n\ge 1. 
\end{equation}

Using a diagonal process, we obtain a subsequence $(n_j)_{j\ge 1}$ such that
$$
v_{0,n_j} \xrightarrow[j\rightarrow \infty]{}v, \qquad \text{uniformly in every compact subset of} \quad \QQ,
$$
where the limit function $v\in W^{1,\infty}_{loc}(\QQ)$ satisfies ~\eqref{Wloc}, hence is differentiable a.e. in $\QQ.$

We now use the stability result for viscosity solutions ~\cite[Th\'eor\`eme~2.3]{Ba94} in order to obtain the fact that $v$ is indeed a viscosity solution to \eqref{HJ}, satisfying ~\eqref{HJ} a.e. in $\QQ.$

Finally, the proof that
$$
v(x,t)\xrightarrow[t\rightarrow {0}^{+}]{}\varphi(x),\qquad \text{uniformly in every compact subset of} \quad \mathbb{R}^N
$$
follows essentially the same reasoning as the corresponding proof in the case of Theorem~\ref{bounded}. As in the case of Theorem~\ref{bounded}, the uniqueness assertion follows from the fact that Equation~\eqref{HJ} satisfies the (discontinuous) comparison principle. This concludes the proof of the theorem.

\begin{proof}[Proof of Lemma ~\ref{influence}]
Let $\xi\in C^\infty_0( \mathbb{R}^N)$ such that $0\leq \xi\leq 1$ and $k$ an integer such that $k>p/(p-1)$. Multiplying ~\eqref{HJV} by $\xi^k$ and integrating over $\mathbb{R}^N$ we get
\begin{eqnarray*}
\frac{d}{dt}\int_{\mathbb{R}^N}\xi(x)^kz(x,t)dx & + & \int_{\mathbb{R}^N}\xi(x)^kH(|\nabla z(x,t)|)dx\\
& = & -\varepsilon\int_{\mathbb{R}^N}\nabla(\xi(x)^k)\cdot\nabla z(x,t)dx.
\end{eqnarray*}

We now use Young's inequality to estimate the right-hand side,
\begin{eqnarray*}
\left| \int_{\mathbb{R}^N}\nabla(\xi(x)^k)\cdot\nabla z(x,t)dx \right| & \leq & \frac{a}{p(p-1)}\int_{\mathbb{R}^N}(\xi(x))^k |\nabla z(x,t)|^pdx\\
& + & \frac{p-1}{p} \left( \frac{a}{p-1} \right)^{-\frac{1}{p-1}} \int_{\mathbb{R}^N}\xi(x)^{-\frac{k}{p-1}}|\nabla(\xi(x)^k)|^{\frac{p}{p-1}}dx,
\end{eqnarray*}
so that by~\eqref{Hass5} we get
$$
\frac{d}{dt}\int_{\mathbb{R}^N}\xi^k(x)z(x,t)dx \leq \frac{p-1}{p} \left( \frac{a}{p-1} \right)^{-\frac{1}{p-1}}\int_{\mathbb{R}^N}\xi(x)^{-\frac{k}{p-1}}|\nabla(\xi(x)^k)|^{\frac{p}{p-1}}dx .
$$
By the choice of $k$ and $\xi$ the first integral in the right-hand side is
\begin{eqnarray*}
\int_{\mathbb{R}^N}\xi(x)^{-\frac{k}{p-1}}|\nabla(\xi(x)^k)|^{\frac{p}{p-1}}dx & = &k^{\frac{p}{p-1}}\int_{\mathbb{R}^N}\xi(x)^{k-\frac{p}{p-1}}|\nabla\xi(x)|^{\frac{p}{p-1}}dx\\
& \leq & k^{\frac{p}{p-1}}\int_{\mathbb{R}^N}|\nabla\xi(x)|^{\frac{p}{p-1}}dx,
\end{eqnarray*}
so that
$$\frac{d}{dt}\int_{\mathbb{R}^N}\xi^k(x)z(x,t)dx\leq \frac{p-1}{p} \left( \frac{a}{p-1} \right)^{-\frac{1}{p-1}}k^{\frac{p}{p-1}}\ \int_{\mathbb{R}^N} \left[ \xi^k(x)+|\nabla\xi(x)|^{\frac{p}{p-1}} \right]dx.$$
Let $\psi\in C^\infty_0 (\mathbb{R}^N)$ be such that $0\leq\psi\leq 1$ and $\psi(x)=1$ (resp. $\psi(x)=0$) if $|x|\leq 1$ (resp. $|x|\geq 2$). Taking $\xi(x)=\psi((x-y)/R)$ in the last estimate and integrating from $s$ to $t$ we obtain ~\eqref{est-ball}.
\end{proof}

%%% ----------------------------------------------------------------------

\begin{appendix}

%%% ----------------------------------------------------------------------

\section{ The case $H(r)=r^p$}\label{apA}

In this Appendix we establish ~\eqref{Hass3} for the special case $H(r)=r^p$, where $p\in (0,\infty),\quad p\neq 1.$ Here we set
\begin{equation}
\label{phim}
\Phim(r) =(r+\eta^2)^{\frac{p}{2}}-\eta^p \quad , \qquad r \in [0, \infty)
\end {equation}
so that  by Definition~\ref{pcondi}
$$\Thetam(r)=(p-1)(r+\eta^2)^{\frac{p}{2}}-p\eta^2(r+\eta^2)^{\frac{p}{2}-1}+ \eta^p, $$

Assume first that $1<p\leq 2.$ Using
       \begin{equation}
       \label{thetam}
       \Thetam(r)\geq(p-1)(r+\eta^2)^{\frac{p}{2}}-(p-1)\eta^p\geq(p-1)[r^{\frac{p}{2}}-\eta^p],
       \end{equation}
Definition~\ref{pcondi}~(i) is established with $a=b=p-1$ and $\gamma=p$.

Consider next the case $p>2.$ Instead of ~\eqref{thetam} we now use the Young inequality to obtain
\begin{eqnarray}
\label{thetam1}
\Thetam(r) & = & (p-1)(r+\eta^2)^{\frac{p}{2}}-p\eta^2(r+\eta^2)^{\frac{p}{2}-1}+\eta^p\\
\nonumber & \geq & (p-1)(r+\eta^2)^{\frac{p}{2}}-\eta (r+\eta^2)^{\frac{p}{2}}-2 (p-2)^{\frac{p-2}{2}} \eta^{\frac{p+2}{2}}.          \end{eqnarray}
Thus, for $p>2$, Definition~\ref{pcondi}~(i) is established with               $a=(p-1)/2$, $b=2\big(p-2\big)^{\frac{p-2}{p}}$ and $\gamma=(p+2)/2$.

The case of a sum of powers ~\eqref{specialpl}
follows immediately from the above argument by
taking $p=\min\set{p_1,...,p_m}.$ Finally we  turn to the case $0<p<1.$ We now have
\begin{eqnarray}
\label{thetam2}
\Thetam(r) & = & pr(r+\eta^2)^{\frac{p}{2}-1}-(r+\eta^2)^{\frac{p}{2}}+\eta^p\\ 
\nonumber
& = & (r+\eta^2)^{\frac{p}{2}}\left[ p-1-\frac{p\eta^2}{r+\eta^2} \right]+\eta^p\\
\nonumber & \leq & (p-1)r^{\frac{p}{2}}+\eta^p,
\end{eqnarray}
which completes the proof of Definition~\ref{pcondi}~(i) for $0<p<1$ with $a=1-p$, $b=1$ and $\gamma=p$. As in the case $p>1,$ this treatment generalizes readily to the case of a sum of powers (all less than $1$) as in ~\eqref{specialpl}, with $p=\max\set{p_1,...,p_m}.$

%%% ----------------------------------------------------------------------

\section{Time equicontinuity}\label{apB}

\begin{prop}\label{pim}
For $t>0$ and $h\in (0,1)$, we have the following estimate.
\begin{equation}
\label{holdert}
\|\veps(\cdot,t+h)-\veps(\cdot,t)\|_{\infty} \leq C_1
h^{\frac{1}{2}}\left\{ \sqrt{\varepsilon}\,
\|\varphi\|_{\infty}^{\frac{1}{p}}t^{-\frac{1}{p}} + Q\left( \lambda_p
\|\varphi\|_{\infty}^{\frac{1}{p}} \,(at)^{-\frac{1}{p}}
\right)\right\} , 
\end{equation}
where $Q(r)=\max\limits_{0\leq s\leq r}H(s),$ and $C_1$ depends only
on $p$, $a$ and $N$.  
\end{prop}

\begin{proof}
To establish this estimate, we consider, for $r>0, \quad y\in
\mathbb{R}^N$ the ball $B_r(y)=\{x \in \mathbb{R}^N,\quad |x-y|<r\}.$
Integrating~\eqref{HJV} over $B_r(y)\times [t,t+h]$ we get (where
$|B_r|$ is the Euclidean volume of the ball) 
\begin{eqnarray*}
(\veps(y,t+h)-\veps(y,t))|B_r|&=&\left[
\int_{B_r(y)}((\veps(y,\tau)-\veps(x,\tau))dx\right]^{\tau=t+h}_{\tau=t}
\\ 
& + &\varepsilon \int_t^{t+h}\int_{\partial
B_r(y)}\nabla\veps\cdot\frac{x}{r}dSd\tau \\ 
& - & \int_t^{t+h}\int_{B_r(y)}H(|\nabla\veps(x,\tau)|)dxd\tau.
 \end{eqnarray*}
Invoking ~\eqref{gradx} we obtain readily
\begin{eqnarray*}
|\veps(y,t+h)-\veps(y,t)| & \leq & C\|\varphi\|_{\infty}^{\frac{1}{p}}
\left\{ \frac{\varepsilon}{r} \left[
(t+h)^{1-\frac{1}{p}}-t^{1-\frac{1}{p}} \right] 
+r \left[ (t+h)^{-\frac{1}{p}}+t^{-\frac{1}{p}} \right] \right\}\\
& + &\int_t^{t+h}Q\left( \lambda_p \|\varphi\|_{\infty}^{\frac{1}{p}}
\,(a\tau)^{-\frac{1}{p}} \right)d\tau, 
\end{eqnarray*}
where $C>0$ depends only on $p$, $a$ and $N$. Taking
$r=\sqrt{\varepsilon h}$ we obtain ~\eqref{holdert}. 
\end{proof}
               
 %%% ----------------------------------------------------------------------

\section{A comparison principle for subadditive and non-decreasing $H$}\label{apC}

\begin{lem}\label{subaddH}
Assume that $H\in C([0,\infty))$ is a nonnegative and non-decreasing
function such that $H(0)=0$ and $H$ is uniformly Lipschitz continuous in $(\delta,\infty)$ for each $\delta>0$. Assume further that $H$ is subadditive, that is,  
$$
H(r+s)\le H(r)+H(s) , \qquad (r,s)\in [0,\infty) .
$$
Then Equation~\eqref{HJ} satisfies the comparison principle in
$C_b(\mathbb{R}^N)$ as stated in Definition~\ref{comparison}~(b).  
\end{lem}

The proof of Lemma~\ref{subaddH} which we give below was kindly
indicated to us by G.~Barles \cite{psbarles}. 

\begin{proof}
Let $v_1\in C_b(\QQ)$ (resp. $v_2\in C_b(\QQ)$) be a viscosity
subsolution  (resp. supersolution) of ~\eqref{HJ} and assume that
$v_1(x,0)\leq v_2(x,0)$ for $x\in\mathbb{R}^N$. We first infer from
the monotonicity and subadditivity of $H$ that  
$$
H(|\xi_1|) + H(|\xi_2-\xi_1|)\ge H(|\xi_1|+|\xi_2-\xi_1|)\ge
H(|\xi_2|) , \quad (\xi_1,\xi_2)\in \mathbb{R}^N\times\mathbb{R}^N. 
$$
Setting $w=v_1-v_2$, it readily follows from the properties of $v_1$,
$v_2$ and the previous inequality that $w$ is a subsolution to 
$\partial_t z - H(|\nabla z|)=0$ in $\QQ$ with $w(\cdot,0)\le 0$. 

Now, on the one hand, if $\delta\in (0,1)$, we have $-H\left( \sqrt{|\xi|^2+\delta^2} \right) \le - H(|\xi|)$ for $\xi\in\mathbb{R}^N$ by the monotonicity of $H$. Consequently, $w$ is also a subsolution to 
\begin{equation}
\label{luchon}
\partial_t z - H\left( \sqrt{|\nabla z|^2 +\delta^2} \right)=0 \quad\mbox{ in  }\quad \QQ
\end{equation}
with $w(\cdot,0)\le 0$. On the other hand, $W_\delta:t\longmapsto H(\delta) t$ clearly solves \eqref{luchon} with $W_\delta(0)=0$ and the Hamiltonian $H_\delta: \xi\longmapsto H\left( \sqrt{|\xi|^2 +\delta^2} \right)$ is uniformly Lipschitz continuous in $\mathbb{R}^N$. We are then in a position to apply \cite[Theorem~V.3]{CL83} and conclude that $w(x,t)\le W_\delta(t)$ for $(x,t)\in\QQ$. Since $H$ is continuous and vanishes at zero, the claimed result then follows by letting $\delta\to 0$. 
\end{proof}

\end{appendix}

%%% ----------------------------------------------------------------------

%%% ----------------------------------------------------------------------

\end{document}